\documentclass[a4paper,12pt]{amsart}

\usepackage{epsfig}
\usepackage{amsthm,amsfonts}
\usepackage{amssymb,graphicx,color}
\usepackage[all]{xy}

\usepackage{showlabels}

\newtheorem{theorem}{Theorem}[section]
\newtheorem{lemma}[theorem]{Lemma}
\newtheorem{corollary}[theorem]{Corollary}
\newtheorem{proposition}[theorem]{Proposition}
\newtheorem{definition}[theorem]{Definition}

\newcommand{\R}{\mathbb R}

\begin{document}

\title[Metrically un-knotted corank 1 singularities of surfaces in $\mathbb{R}^4$
 ]{Metrically un-knotted corank 1 singularities of surfaces in $\mathbb{R}^4$.
}

\author{L. Birbrair}

\address{Departamento of Matem\'atica, 
Universidade Federal do Cear\'a, Campus do Pici Bloco 914, CEP 60455--760, Fortaleza CE, Brazil}

\email{lev.birbrair@gmail.com}

\author[R. Mendes]{Rodrigo Mendes}
\address{Departamento de Matem\'atica, Universidade de Integra\c{c}\~ao Internacional da Lusofonia Afro-Brasileira (unilab)
, Campus dos Palmares, Cep. 62785-000. Acarape-Ce,
Brasil} \email{rodrigomendes@unilab.edu.br}

\author{J.J.~Nu\~no-Ballesteros}

\address{Departament de Matem\`atiques,
Universitat de Val\`encia, Campus de Burjassot, 46100 Burjassot,
Spain}

\email{Juan.Nuno@uv.es}


\begin{abstract}
The paper is devoted to relations between topological and metric properties of germs of real surfaces, obtained by analytic maps from $\mathbb{R}^2$ to $\mathbb{R}^4$. We show that for a big class of such surfaces the normal embedding property implies the triviality of the knot, presenting the link of the surfaces. We also present some criteria of normal embedding in terms of the polar curves.
\end{abstract}

\maketitle

\section{Introduction}
 
This paper is devoted to topological and metric properties of surfaces, defined {as the image of a polynomial map from $\mathbb{R}^2$ to $\mathbb{R}^4$ with isolated singularities}. The link of these surfaces {at a singular point} can be considered as {a knot} in $\mathbb{S}^3$. The main question is the following: determine the relation {between} the metric properties of surfaces and the topological properties of knots. Here, we make the first steps of this theory. {We recall} that the surface is called locally normally embedded {at a point} if the inner metric is locally bi-Lipschitz equivalent to the outer metric. {An interesting particular case is when the surface is a complex plane curve, that is, it is parametrized by a complex polynomial map from $\mathbb{C}$ to $\mathbb{C}^2$.} It is proved by Fernandes \cite{F} that a {complex plane curve} is locally normally embedded at a point $x_0$ if and only if $x_0$ is a smooth point. The main result of the paper is the following: {if a surface germ $(X,0)$ is normally embedded and is parametrized by a map germ $F:(\mathbb{R}^2,0)\to(\mathbb{R}^4,0)$ whose $2$-jet is equivalent to $(x,xy,0,0)$, then the link of $(X,0)$ is a trivial knot.}

The proof of the triviality of the {knot of a normally embedded surface} is based on the height-width property of the polar set of the projection of the surface to its tangent cone. This criterion was discovered in the PhD thesis (see \cite{RM}) of Rodrigo Mendes (the second author of the paper). Recently a very similar criterion was announced by Leslie Wilson and Donald O'Shea in \cite{LD}.

\bigskip

\noindent{\bf Acknowledgements} We would like to thank Alexandre Fernandes, Vincent Grandjean and Edson Sampaio for interesting discussions and important remarks.

\section{Basic definitions}
Let $X \subset \mathbb{R}^n$ be a closed and connected {semialgebraic} set. The \emph{inner metric} on $X$  is defined as follows: given two points $x_1,x_2\in X$, $d_{inner}(x_1,x_2)$  is the infimum of the lengths of rectifiable paths on $X$ connecting $x_1$ to $x_2$. {The set} $X$ is called \emph{normally embedded} if there exists $\lambda >0$ such that 
$$d_{inner}(x_1,x_2)\le \lambda \|x_1-x_2\|,$$
for all $x_1,x_2\in X$.

A semialgebraic germ $(X,x_0) \subset (\R^n,x_0)$ is called \emph{normally embedded} when there is a {normally embedded} representative $X \cap \mathcal{U}$, where $\mathcal{U}\subset \mathbb{R}^n$ is a open set containing $x_0$.
\begin{theorem} (Arc criterion \cite{MB}) \emph{{Let $(X,x_0) \subset (\R^n,x_0)$} be a semialgebraic germ. Then the following assertions are equivalent:}
\begin{itemize}
\item \emph{$(X,x_0)$ is normally embedded;}
\item \emph{There exists a constant $k>0$ such that for any pair of {arcs} $\gamma_1, \gamma_2$  parametrized by the distance to $x_0$ ($\gamma_i(0)=x_0$), we have }
\begin{center}
$d_{inner}(\gamma_1(t), \gamma_2(t))\leq k\|\gamma_1(t)-\gamma_2(t))\|$.
\end{center}
\end{itemize}
\end{theorem}

{Now we recall other definitions we need in the next sections.}

\begin{definition} 
{\rm Let $\gamma_1$ and $\gamma_2$ be germs of semialgebraic {arcs} at $x_0 \in \mathbb{R}^n$. {Assume} that the two arcs are parametrized by the distance to$x_0$. The function ${d_{out}}(\gamma_1(t),\gamma_2(t))=\|\gamma_1(t)-\gamma_2(t)\|$ admits a Newton-Puiseux expansion and the {smallest} exponent is called the (extrinsic) \emph{order of contact} of $\gamma_1$ and $\gamma_2$ {and is denoted by $tord(\gamma_1,\gamma_2)$.}

By the results of Kurdyka and Orro or Birbrair and Mostowski (see \cite{K} and \cite{BM}) there exists a semialgebraic metric $d_P$ bi-Lipschitz equivalent to the intrinsic metric $d_{inner}$. That is why we can {also} define the (intrinsic) \emph{order of contact} of $\gamma_1$ and $\gamma_2$ as the order of the function $d_{inner}(\gamma_1(t),\gamma_2(t))$, {which is denoted  by $tord_{inn}(\gamma_1,\gamma_2)$.}}
\end{definition}

 If $X$ is normally embedded at $x_0$, then: 
\begin{center}
$tord(\gamma_1,\gamma_2)=ord_t\|\gamma_1(t)-\gamma_2(t)\|=ord_td_{inner}(\gamma_1(t),\gamma_2(t))=tord_{inn}(\gamma_1,\gamma_2)$,
\end{center}
for any pair of arcs $\gamma_1, \gamma_2$ in $X$ with $\gamma_i(0)=x_0$.

\begin{definition} 
\emph{A subset $X \subset \mathbb{R}^n$ is called a $\beta$-\emph{H\"older triangle} with vertex $x_0 \in X$ if the germ $(X,x_0)$ is bi-Lipschitz equivalent with respect to the outer metric to the germ $(T_{\beta},0)$, where}
\begin{center}
$T_{\beta}=\{(x,y) \in \R^2; 0 \leq y \leq x^\beta, 0 \leq x \leq 1 \}$,
\end{center}
{and} $\beta \in \mathbb{Q} \cap [1,\infty)$.
\end{definition}

\newpage

\begin{definition}
\emph{{Let $X\subset \R^n$ be a semialgebraic set.} The \emph{tangent cone} $T_{x_0}X$ of $X$ at $x_0$ is the closed cone over the subset $S_{x_0}X$,  the \emph{tangent link}, defined as}
\begin{center}
$S_{x_0}X=\{v \in \mathbb{R}^n; v=\lim_{y_j \to x_0}\frac{y_j-x_0}{\|y_j-x_0\|}\},$
\end{center}
\emph{where $y_j \in X-\{x_0\}$ is a sequence converging to $x_0$. } 
\end{definition}

Notice that the tangent cone can be also defined as the set of tangent vectors to all the semialgebraic arcs passing through $x_0$.

\begin{definition} 
{\rm Let $F:(\mathbb{R}^n,0) \rightarrow (\mathbb{R}^p,0)$ be a polynomial map germ ($n\leq p$). 

\begin{itemize}
\item {The \emph{corank} of $F$ is the dimension of the kernel of the differential $DF(0)$.}

\item For each variable $x_i$, consider the function $g(x_i)=F(0,\ldots,0,x_i,0,\ldots,0)$. {If $g\ne0$}, we have
\begin{center}
$g(x_i)=x_i^m\tilde{g}(x_i),$
\end{center}
for some $m \in \mathbb{N}$ where $\tilde{g}(0)\neq 0$. Then $m$ is called the \emph{order} of $F$ with respect to the variable $x_i$ and is denoted by $m=ord_{x_i}F(0,\ldots,0,x_i,0,\ldots,0)$. {If $g=0$, then we set $ord_{x_i}F(0,\ldots,0,x_i,0,\ldots,0)=\infty$.}

\end{itemize}}

\end{definition}

\section{Tangent Cone and corank 1 singularities}

\begin{theorem}
\emph{Let $F:(\mathbb{R}^2,0)\rightarrow (\mathbb{R}^4,0)$ be an injective {polynomial map germ} of corank 1 and {let $(X,0)$ be its image}. Then, the tangent cone of $X$ at $0 \in \mathbb{R}^4$ is {either} a two-dimensional plane $V \subset \R^4$ or a half-plane}.
\end{theorem} 
\begin{proof}
{After an analytic coordinate change in $(\mathbb{R}^2,0)$ and  a linear coordinate change in $\mathbb{R}^4$}, the map $F$ can be presented in the following way:
\begin{equation}
F(x,y)=(x,F_2(x,y),F_3(x,y),F_4(x,y)), 
\end{equation} 
where $ord_y(F_2(0,y))<ord_y(F_3(0,y))\leq ord_y(F_4(0,y))$ and $F_2(x,0)=F_3(x,0)=F_4(x,0)=0$. 

Notice that the tangent cone can be considered as the set of all the tangent vectors of all analytic arcs $\gamma \subset X$, such that $\gamma(0)=0$. Any analytic arc $\gamma$ of this form has a pull-back $\tilde{\gamma}\subset \mathbb{R}^2$, such that $\tilde{\gamma}(0)=0$. These arcs can be presented in the form $(x,f(x))$ or $(g(y),y)$, where $f$ and $g$ are some subanalytic functions, such that $ord_x(f)\geq 1$ and $ord_y(g)\geq 1$. The tangent vector of $\gamma$ is determined by the smallest Puiseux exponent. {Let $F_2(0,y)=a y^p+o(y^p)$ where $a\ne0$ and $o(y^p)$ means terms of order $>p$} . Consider the arcs of the form $(by^p,y)=\tilde{\gamma}(y)$, $b \in \mathbb{R} \setminus \{0\}$. We have 
$$
F(by^p,y)=(by^p,ay^p+o(y^p),o(y^p),o(y^p)),
$$

Computing the tangent vector, we obtain $(b,a,0,0)$. If $p$ is even, then the tangent vector of these arcs generate a half-plane and if $p$ is odd, then they generate a two dimensional plane. The tangent vectors to the other arcs are parallel to the $x$-axis or the $y$-axis. Namely, if $\tilde{\gamma}(x)=(x,cx^q)$, then $F(\tilde{\gamma}(x))=(x,o(x^q),o(x^q),o(x^q))$ and the tangent vector is equal to $(1,0,0,0)$. {On} the other hand, if $\tilde{\gamma}(y)=(by^q,y)$ and $q\neq p$, we obtain the tangent vector $({b},0,0,0)$ if $q<p$ or $(0,{a},0,0)$ if $q>p$.
\end{proof}
\begin{definition}
\emph{For a surface $X \subset \mathbb{R}^4$ (i.e. closed, connected 2-dimensional semialgebraic set), such that the tangent cone $T_0X$ is contained in a plane $\mathbb{R}^2_{(T)}$, one can define the polar set of $X$ and the discriminant set of $X$ as follows:  The polar set of $X$ is the set of the points $x \in X$, such that the derivative $dP|_X$ is not an isomorphism, where $P:X \rightarrow \mathbb{R}^2_{(T)}$ is the orthogonal projection. The discriminant set is defined as the image of the polar set by this projection. Notice that, in {the semialgebraic case}, this set is one-dimensional}.
\end{definition}

\begin{theorem}
\emph{If the tangent cone of $X$ at $x_0$ {is} a half-plane, then the germ {$(X,x_0)$} is not normally embedded}.
\end{theorem}

\begin{proof}

Consider the orthogonal projection $P:X \rightarrow \mathbb{R}^2_{(T)}$, where $\mathbb{R}^2_{(T)}$ is the plane containing the tangent cone $T_{x_0}X$. {The}  inverse image of any point by the map $P$ is a finite set (otherwise the tangent cone cannot {be contained in} this plane). Then one can define the degree of this map. {But} the degree is equal to zero, because the map is not surjective. Consider the polar set of $X$ {and let $l$ be a line in $\mathbb{R}^2_{(T)}$} which is not tangent to any component of the discriminant set. Since the degree of the map is equal to zero, then the germ of line $l$ has at least two {inverse image arcs} $\gamma_1$ {and} $\gamma_2$. We parametrize these arcs by the distance to the origin, {so that} $\|\gamma_i(t)\|=t$. Since the line $l$ is not tangent to the components of the discriminant set one has $tord_{inn}(\gamma_1,\gamma_2)=1$. {On} the other hand, the arcs $\gamma_1$ and $\gamma_2$ are tangent to the line $l$, {hence}
\begin{center}
$tord(\gamma_1,\gamma_2)>1$.
\end{center}
It means that  the set $X$ is not normally embedded at $x_0$ by the arc criterion. 
\end{proof}

\begin{corollary}
\emph{Let {$(X,0)$ be a surface parametrized as the image of a corank 1 map germ $F:(\mathbb{R}^2,0)\to(\mathbb{R}^4,0)$ such that} $ord_yF(0,y)$ is even. Then, $(X,0)$ is not normally embedded}.
\end{corollary}

\section{Height and width of surfaces in $\mathbb{R}^4$}
Let $(X,0)$ be a germ of a surface such that $T_0X$ is a subset of a 2-dimensional plane $\mathbb{R}^2_{(T)}$. Let us consider  $P:X \rightarrow \mathbb{R}^2_{(T)}$ {the} orthogonal projection. 

\begin{proposition}
\emph{ {Assume $(X,0)$ is normally embedded}. A subset $T'\subsetneq T_0X$ such that $\forall v \in T'$ one has $\#(X \cap P^{-1}(v))=1$ has the following structure:
\begin{center}
$T'=V-\cup T_{\beta_i}$,
\end{center}
where $T_{\beta_i}$ are H\"older triangles with $\beta_i>1$}.
\end{proposition} 
\begin{proof}
Notice that the discriminant set of the projection is a union of semi-analytic arcs and the function $\#(X \cap P^{-1}(v))$ is locally constant on the complement of the discriminant set.  Now, we must prove that the zones where $\#(P^{-1}(v) \cap X)>1$ are "$thin$" (i.e are H\"older triangles with exponent bigger than one). Suppose {this is not true}. It means that there is a triangle $T_i$, bounded by two arcs $\gamma_1$ and $\gamma_2$ such that $\gamma_1$ and $\gamma_2$ are not tangent at zero and $\#(P^{-1}(v)\cap X)>1$ on $T_i$. Take an arc $\beta \subset T_i$, passing through zero and not tangent neither to $\gamma_1$, nor to $\gamma_2$ such that {$P ^{-1}(\beta) \cap X$ contains two different arcs $\tilde{\beta_1}, \tilde{\beta_2}$}.  Then $\tilde{\beta_1}, \tilde{\beta_2}$ and $\beta$ are tangent at zero. Let us parametrize $\tilde{\beta_1}$ and $\tilde{\beta_2}$ by the distance to zero (i.e., $\|\tilde{\beta_i}(t)\|=t$). Then {$tord_{inn}(\tilde{\beta}_1,\tilde{\beta}_2)=1$} because any path, connecting these two point must go through $P^{-1}(\gamma_1)$ or trough $P^{-1}(\gamma_2)$. {But this is not possible if $(X,0)$ is normally embedded.} 
\end{proof}

Now, let $X$ be any surface {whose tangent cone is} a plane at $x_0 \in X$.

Let $T_i$ be a H\"older triangle whose boundary is contained in the discriminant set and $\#(P^{-1}(v)\cap X)>1$ on $T_i$. The triangles $T_i$, bounded by the discriminant set and such that for each $v \in T_i$ one has $\#(P^{-1}(v)\cap X)>1$ are called the \emph{polar triangles}. Let us associate two rational numbers to the triangle $T_i$. The first number  associated to $T_i$ is called the \emph{width} of $T_i$, and is defined as the order of contact of the boundary arcs $\gamma_1$ and $\gamma_2$, i.e., width$(T_i)=tord(\gamma_1,\gamma_2)$.  Another rational number is called the \emph{height} of $T_i$ and is defined as
\begin{center}
height$(T_i)=min \{ tord(\tilde{\beta}_1(t),\tilde{\beta}_2(t)):\; \tilde{\beta_1},\tilde{\beta_2} \subset P^{-1}(\beta)\}$.
\end{center}
where $\tilde{\beta_1}, \tilde{\beta_2}$ are semi-arcs in $X$, $\beta \subset T_i$.

\begin{theorem}
\emph{If $(X,x_0)$ is normally embedded and the set of polar triangles is non-empty, then for all $T_i$, one has }
\begin{center}
$width(T_i) \geq height(T_i)$.
\end{center}

\end{theorem}

\begin{proof}

Let us make a Fukuda-like reparametrization of our projection map $g:X \rightarrow \R^2$, $g(x)=\|x\|\frac{P(x)}{\|P(x)\|}$. {Then} the polar and the discriminant set of {$P$} is the same as the polar and the discriminant {set} of $g$. Suppose that $width(T_i)<height(T_i)$. The set $T_i$ can be divided into subsets $T_i= \cup T_{i,s}$, such that for all $s$ one has that $\#\{P^{-1}(v), v \in T_{i,s}\}$ is constant and there {are} no discriminant curves belonging to $int(T_{i,s})$, {that is}, $T_{i,s} \cap \Delta =\partial T_{i,s}$, {where $\Delta$ is the discriminant set}. By the valuation (isosceles) property, there exists $s$ such that 
\begin{center}
$ord_t\|\gamma_{1,s}(t)-\gamma_{2,s}(t)\|=ord_t\|\gamma_1(t)-\gamma_2(t)\|$,
\end{center}
where $\gamma_{1,s},\gamma_{2,s}$ are boundary curves of $T_{i,s}$. Take an arc $\gamma \subset T_{i,s}$ such that 
\begin{center}
$ord_t\|\gamma(t)- \gamma_{1,s}(t)\|=ord_t\|\gamma(t)-\gamma_{2,s}(t)\|=ord_t|\gamma_1(t)-\gamma_2(t)\|=width(T_i)$.
\end{center}
{By the cone structure theorem (see \cite{MJB}), there exists $\epsilon>0$ such that $X\cap B_\epsilon(x_0)$} can be represented in polar coordinates $(\theta,t)$, where $\theta \in \mathbb{S}^1$ and $t \in (0,\epsilon)$, {and} $\mathbb{S}^1$ {represents the} link of the surface $X$ {at $x_0$}. For each $t$ there are at least {a pair} of segments $[\theta_1(t),\theta_2(t)]$ and $[\theta_3(t),\theta_4(t)]$ in $\mathbb{S}^1$ such that 
\begin{itemize}
\item $g([\theta_1(t),\theta_2(t)]) \subset T_{i,s} \  \ 
       g([\theta_3(t),\theta_4(t)]) \subset T_{i,s}$
 \item $g(\theta_j(t))\subset \partial T_{i,s}$ for all $j=1,2,3,4$
 \item $g(\theta_1(t)) \neq g(\theta_2(t)), \ \  g(\theta_3(t)) \neq g(\theta_4(t))$.
\end{itemize}
We can choose one of these segments realizing $width(T_i)$ (say $[\theta_1(t),\theta_2(t)]$). Hence on each segment $[\theta_1(t),\theta_2(t)]$, $[\theta_3(t),\theta_4(t)]$ there exists a point belonging to an inverse image of $\gamma$. Consider two components $\tilde{\gamma_1},\tilde{\gamma_2}$ of $g^{-1}(\gamma)$ belonging two different H\"older triangles $\cup_{t}[\theta_1(t),\theta_2(t)]$ and $\cup_{t}[\theta_3(t),\theta_4(t)]$ . Then 
$$ord(d_{inn}( \tilde{\gamma_1}(t),\tilde{\gamma_2}(t))) \leq ord\| \theta_1(t)-\tilde{\gamma_1}(t)\| \leq width(T_{i,s}).$$
But, $ord\| \tilde{\gamma_1}(t)-\tilde{\gamma_2}(t)\| \geq height(T_{i,s}) \geq height(T_i)>width(T_i)=width(T_{i,s})$. {This implies that} $X$ cannot be normally embedded.
\end{proof}

\section{Normally embedded and corank 1 singularities}
In \cite{MJB}, the authors show that for maps $F: (\mathbb{R}^2,0)\rightarrow (\mathbb{R}^4,0)$ such that $corank(F)=1$, there are four orbits {in the 2-jet space with respect to the left-right action. That is, after 2-jets of diffeomorphisms in the source and target, we have that $J^2(F)$ is equivalent to}
$$
(x,y^2,xy,0),\quad (x,y^2,0,0),\quad (x,xy,0,0),\quad (x,0,0,0).
$$
By theorem 3.3, if $J^2(F)=(x,y^2,xy,0)$ or $J^2(F)=(x,y^2,0,0)$ then $(X,0)=F(\mathbb{R}^2,0)$ is not normally embedded. 

\medskip

{We recall that if $(X,0)$ has isolated singularity, then}  for small $\epsilon_0>0$, the link $X\cap \mathbb{S}^3(0,\epsilon)$ determines a knot. Moreover, $\{X\cap \mathbb{S}^3(0,\epsilon); 0<\epsilon\leq \epsilon_0\}$ are all isotopic by the conic structure (\cite{MJB}, theorem 2.5). We say that $X$ is topologically locally flat at $0$ (or that $0$ is non-singular point) when $X\cap \mathbb{S}^3(0,\epsilon)$ is a trivial knot. Actually, we have $X$ is locally flat in the sense of Fox and Milnor (see \cite{FM}). In this direction, we have the following result:
\begin{theorem}
\emph{Let $(X,0)=F(\mathbb{R}^2,0)$ be surface with isolated singularity, where $J^2(F)=(x,xy,0,0)$. If $X$ is normally embedded at $0$, then $X$ is locally flat at $0$.}
\begin{proof}
We need some preliminary comments. An orthogonal projection $P:\mathbb{R}^4 \rightarrow \mathbb{R}^3$ is called stable with respect to $X$ if all the singularities of $P(X)$ are {transverse} double points (In the sense of \cite{WD}). Notice that the property of being stable is a generic property. It means that the set of kernels of the non-stable projections is a subanalytic set of codimension bigger or equal to $1$ in $\mathbb{RP}^3$. We need the following lemma:
\begin{lemma}
\emph{Let $(X,0)=F(\mathbb{R}^2,0) \subset (\mathbb{R}^4,0)$ be a germ of  parametrized surface, where $F$ has corank $1$ in $0$. Then, there exists a linear change of variables such that the projection
\begin{center}
$P(z_1,z_2,z_3,z_4)=(z_1,z_2,z_3,0)$
\end{center}
is a stable projection with respect to $X-\{0\}$ and the image of the tangent cone of $X$ at zero by this projection is the tangent cone to $P(X)$. Moreover, the coordinate functions satisfy the conditions} (1).
\end{lemma}
\begin{proof}

Choose a direction of a projection in $\mathbb{RP}^3$, such that this direction is transversal to the tangent cone (plane) of $X$ in $0$ and such that the projection is stable. We may  choose the coordinates $z_1,z_2,z_3,z_4$, such that the tangent cone {is the plane} $(z_1,z_2,0,0)$. Then, {$F$ may be written in these coordinates as follows:}
\begin{center}
$(x,F_2(x,y),F_3(x,y),F_4(x,y))$,
\end{center}
where $ord_y(F_2(0,y)) < ord_y(F_3(0,y))$ and $ord_y(F_2(0,y)) < ord_y(F_4(0,y))$. Moreover, $F(x,0)=(x,0,0,0)$. Notice, that this change of coordinates {does} not destroy the stability of the projection.
\end{proof}

Proof of theorem: 

Let $X$ be normally embedded set at $0$, by the previous corollary, we have $T_0X$ is a plane {and} $ord_yF(0,y)$ is odd. {Moreover,} $F$ may be written as follows:

\begin{center}
$(x,xy+P_1(x,y),Q(x,y), R(x,y))$,
\end{center}
where $P_1, Q, R$ are elements of the ideal $\mathcal{M}_2^3=(x^3,x^2y,xy^2,y^3)$.

\medskip

Claim 1. {Since} $X$ is normally embedded in $0$, then $ord_yF(0,y)=ord_yQ(0,y)$ or $ord_yF(0,y)=ord_yR(0,y)$.

\medskip
\begin{proof}
Suppose {it is not true}. It {means} that, {up to} a linear change of coordinates, $ord_yF(0,y)=ord_yP(0,y)<ord_yQ(0,y)<ord_yR(0,y)$.

We define the polar set of the surface $X$ with respect to the parametrization $F$ as follows: Consider the map $P_T \circ F: \mathbb{R}^2 \rightarrow \mathbb{R}^2$, where $F$ is the parametrization and $P_T$ is the orthogonal projection of $X$ to $T_0X$, where  $P_T(z_1,z_2,z_3,z_4)=(z_1,z_2,0,0)$, by  lemma 5.2. Let $\Sigma$ be the singular locus of the map $P_T \circ F=(x,xy+P_1)$, {that is,}
\begin{center}
$\Sigma=\{(x,y) \in \mathcal{U}; x+\frac{\partial P_1(x,y)}{\partial y}=0\}$.
\end{center}
{Then} $\Sigma$ is a smooth curve (has only 1 branch) and admits {the} parametrization  ${\sigma(t)}=(ct^{n-1}+o(t^{n-1}),t)$, $n=ordF(0,y), \ c \in \mathbb{R}$. Observe that the discriminant of the projection $P_T:X\rightarrow T_0X$ is given by
\begin{center}
$\Delta=P_T\circ F(\sigma(t))=P_T \circ F(ct^{n-1}+o(t^{n-1}),t)=(Ct^{n-1}+o(t^{n-1}),dt^n+o(t^n))$.
\end{center}
{Thus,} $\Delta$ is a real cusp, $0 \in \Delta \subset T_0X$. Consider $T_{\Delta}$ be a H\"older triangle such that the boundary of $T_{\Delta}$ is $\Delta$. Observe that $width(T_{\Delta})=\frac{n}{n-1}$. 

\medskip

Notice that $height(T_\Delta)$ is well defined.
In fact, without loss of generality, we can suppose that $l_1=\{(x,0); x \geq 0\}$ is contained in $T_{\Delta}$. {Then} $l_1$ has two inverse images: $(x,0,0,0)$ and $\gamma_x=F(\tilde{\gamma_x})$, where $\tilde{\gamma_x}\neq (x,0)$ is {in the} zero locus of $\frac{xy+P_1(x,y)}{y}$.

\medskip
Now, it is sufficient to prove that $height(T_\Delta)>width(T_\Delta)$. 

\medskip

{To see} this, we need to estimate $height(T_{\Delta})$. Consider semialgebraic arcs $\beta_1,\beta_2$ such that $P_T(\beta_1)=P_T(\beta_2)=\beta \subset T_\Delta$ and $\beta_1,\beta_2$ realize $height(T_{\Delta})$, i.e., $height(T_{\Delta})=tord(\beta_1,\beta_2)$. Let $\tilde \beta_1$ be a pull back of $\beta_1$ and $\tilde \beta_2$ be a pull back of $\beta_2$, i.e., $F(\tilde{\beta}_1)=\beta_1$ and $F(\tilde{\beta}_2)=\beta_2$. Since $F(x,y)=(x,P,Q,R)$, we can suppose that $\beta, \beta_1, \beta_2$ are parametrized by the $x$-axis. Then, $tord(\beta_1,\beta_2)$ can be calculated as follows:
$$
ord_t\|F\circ \tilde{\beta}_1(t)-F \circ \tilde{\beta}_2(t)\|=ord_t\|(P_T \circ F(\tilde{\beta}_1),Q(\tilde{\beta}_1),R(\tilde{\beta}_1))-(P_T \circ F(\tilde{\beta}_2),Q(\tilde{\beta}_2),R(\tilde{\beta}_2))\|.
$$

It follows that:
$$
tord(\beta_1,\beta_2)=\|Q(\tilde{\beta}_1)-Q(\tilde{\beta}_2), R(\tilde{\beta}_1)-R(\tilde{\beta}_2)\|.
$$
In {other words}, we have $tord(\beta_1,\beta_2)=$ min$\{ord_t\|Q(\tilde{\beta}_1)-Q(\tilde{\beta}_2)\|, ord_t\|R(\tilde{\beta}_1)-R(\tilde{\beta}_2)\|\}$.

\medskip

Without loss of generality, we assume that $tord(\beta_1,\beta_2)=\|Q(\tilde{\beta}_1)-Q(\tilde{\beta}_2)\|$. We have the exponents $\alpha_1, \alpha_2$ given by 
$$
\alpha_i=tord(\beta_i,\beta)=ord_t\|(P_T \circ F(\tilde{\beta}_i),Q(\tilde{\beta}_i),R(\tilde{\beta}_i))-P_T \circ F(\tilde{\beta}_i)\|=ord_t\|(Q(\tilde{\beta}_i),R(\tilde{\beta}_i))\|,
$$
$i=1,2$. By the non archimedian property (see \cite{BF}), $tord(\beta_1,\beta_2)\geq tord(\beta_i,\beta)$, for some $i$, say, $i=1$. Then,
$$
ord_t\|Q(\tilde{\beta}_1)-Q(\tilde{\beta}_2)\|\geq ord_t\|(Q(\tilde{\beta}_1),R(\tilde{\beta}_1))\|.
$$

{We have that} $\tilde{\beta}_1$ is contained in $T(\Sigma)$ (the H\"older triangle with boundary $\Sigma$), where $P_T(T(\Sigma))=T_\Delta$. 

If $\tilde{\beta}_1$ is not tangent to the $y$-axis, we have a parametrization given by $\tilde{\beta}_1(t)=(et,t^{\tilde{\alpha}_i}+o(t^{\tilde{\alpha}_i}))$, $\tilde{\alpha}_i \geq 1$. In this case, $\beta_i$ and $\beta$ are already parametrized by the tangent vector $(1,0,0,0)$. Then, since $Q, R \in \mathcal{M}_2^3$, we have $tord(\beta_1,\beta_2)>ord_t\|(Q(\tilde{\beta}_1(t)),R(\tilde{\beta}_1(t)))\|>3>\frac{n}{n-1}$.

\medskip

\item If $\tilde{\beta}_1$ is tangent to $y$-axes, we have a parametrization $\tilde{\beta}_1(t)=(ct^{\tilde{\alpha}}+o(t^{\tilde{\alpha}}),t)$ with $\tilde{\beta}_1 \subset T(\Sigma)$. Then, $1\leq \tilde{\alpha}\leq n-1$ and $\beta_1(t)=(ct^{\tilde{\alpha}}+o(t^{\tilde{\alpha}}),P(\tilde{\beta}_1(t)), Q(\tilde{\beta}_1(t)),R(\tilde{\beta}_1(t)))$. Take the reparametrization
$$
\overline{\beta}_1(t)=\beta_1(t^{\frac{1}{\tilde{\alpha}}})=(ct+o(t),P(\tilde{\beta}_1(t^{\frac{1}{\tilde{\alpha}}}), Q(\tilde{\beta}_1(t^{\frac{1}{\tilde{\alpha}}}), R(t^{\frac{1}{\tilde{\alpha}}})).
$$
Notice that $ord_t\|(Q(\tilde{\beta}_1),R(\tilde{\beta}_1))\| \geq \min\{ \frac{ordQ(0,t)}{\tilde{\alpha}},\frac{2\tilde{\alpha}+1}{\tilde{\alpha}},\frac{2+\tilde{\alpha}}{\tilde{\alpha}}\} > \frac{n}{n-1}=width(T_\Delta)$. In fact, 
\begin{itemize}
\item $ordQ_y(0,y)>ord_yP(0,y) \Rightarrow \frac{ordQ(0,t)}{\tilde{\alpha}}> \frac{n}{n-1}$;

\item $\frac{2\tilde{\alpha}+1}{\tilde{\alpha}}=\frac{1}{\tilde{\alpha}}+2\geq 2+\frac{1}{n-1}> \frac{n}{n-1}$;
\item $\frac{2+\tilde{\alpha}}{\tilde{\alpha}}=\frac{2}{\tilde{\alpha}}+1 \geq 1+\frac{2}{n-1}>\frac{n}{n-1}$;
\end{itemize}
Thus, we conclude that $height(T_\Delta)>width(T_\Delta)$ and by theorem 4.2, $X$ is not normally embedded at $0$.
\end{proof}

End of the proof of the theorem: 

\medskip

Let $q=ord_yQ(0,y)$ and $p=ord_yP(0,y)$. {After a} linear change of coordinate, we can {assume} $ordF_y(0,y)=ord_yQ(0,y)<ord_yP(0,y)<ord_yR(0,y)$. In this case, consider $P_3:\mathbb{R}^4 \rightarrow \mathbb{R}^3$, $P_3(z_1,z_2,z_3,z_4)=(z_1,z_2,z_3,0)$ being a generic projection as in lemma \ref{Lema}. The set $P_3(X)$ provides the information on generic diagram of the knot $X \cap \mathbb{S}^3(0,\epsilon)$, when we consider $P_3(X) \cap \mathbb{S}^2(0,\epsilon)$, $\epsilon$ sufficiently small. The transverse double points of the diagram of the knot with respect to this projection can be obtained from the equations (see \cite{WJ}, \cite{WD}):

$$
\Delta P(x,y,u)=\frac{P(x,u)-P(x,y)}{u-y}=\Delta Q(x,y,u)=\frac{Q(x,u)-Q(x,y)}{u-y}=0 \Leftrightarrow
$$
$$
x+a(y^{p-1}+\ldots+u^{p-1})+\ldots =0
$$
 and $$
a_1(x)+a_2(x)(y+u)+a_3(x)(y^2+uy+u^2)+\ldots+b(y^{q-1}+y^{q-2}u+\ldots+yu^{q-2}+u^{q-1})+\ldots=0,
$$
where
 $$
 Q(x,y)=a_1(x)y+a_2(x)y^2+\ldots+a_q(x)y^q+\ldots
 $$
 with $a_q(0)\neq 0$.
{We are interested in the initial terms of this equation. To compute this, we set $H_{q-1}(y,u)=y^{q-1}+y^{q-2}u+\ldots+yu^{q-2}+u^{q-1}$. Then, we have $x=-aH_{p-1}(y,u)$
and $$
\Delta Q(-aH_{p-1}(y,u),y,u)=a_1(-aH_{p-1}(y,u))+a_2(-aH_{p-1}(y,u))(y+u)+\ldots+bH_{q-1}(y,u)+\ldots=0.
$$ }
Since $q=ord_yQ(0,y)<ord_yP(0,y)=p$, the initial part of this equation is 
\begin{center}
$H_{q-1}(y,u)=y^{q-1}+y^{q-2}u+\ldots+yu^{q-2}+u^{q-1}=0$.
\end{center}
{But} $q-1$ is even, {so} this equation {does} not have real solutions $\neq (0,0,0)$. Hence, the projection $P_3(X)$ {does} not have double points and, thus, $X\cap \mathbb{S}^3(0,\epsilon)$ is a trivial knot.

\end{proof}
\end{theorem}


\begin{thebibliography}{99}

\bibitem {B} L. Birbrair, Local  bi-Lipschitz  classification  of  2-dimensional  semialgebraic  sets. Houston Journal of Mathematics, N3, vol.25, (1999), pp 453-472.

\bibitem{K} K. Kurdyka and P. Orro, {\em Distance g{\'e}od{\'e}sique sur un sous-analytique}, Revista Mat. Univ. Complutense de Madrid {\bf 10} (1997), no. Suplementario, 173--182.

\bibitem{BM} Birbrair, L. and Mostowski, T. {\it Normal embeddings of semialgebraic sets}, Michigan Math. J. {\bf 47} (2000), 125-132
 


\bibitem{BF}  Birbrair, Lev; Fernandes, Alexandre C. G. {\it Metric theory of semialgebraic curves}, Rev. Mat. Complut. {\bf 13} (2000), no. 2, 369–-382.

\bibitem{MB} Birbrair, L; Mendes, R. {\it Arc criterion of normal embedding}, Advances in Mathematics, Springer, to appear.

\bibitem{F}  Fernandes, Alexandre Topological equivalence of complex curves and bi-Lipschitz homeomorphisms. Michigan Math. J. 51 (2003), no. 3, 593–606.

\bibitem{RM} Mendes, R. Geometria m\'etrica e topologia de superf\'icies algebricamente parametrizadas. Tese de doutorado, UFC, 2016.

\bibitem{WD} W.L. Marar, D. Mond, Multiple point schemes for corank 1 maps, J. London Math. Soc. 39 (1989) 553–567.

\bibitem{WJ} W.L. Marar and J.J. Nu\~no-Ballesteros, The doodle of a finitely determined map germ  from $\R^2$ to $\R^3$, Adv. Math. {\bf 221} (2009), 1281-1301.

\bibitem{MJB}  Mendes, R.; Nu\~no-Ballesteros, J. J. Knots and the topology of singular surfaces in $\mathbb{R}^4$,  Real and complex singularities, 229--239, Contemp. Math., 675, Amer. Math. Soc., Providence, RI, 2016.

\bibitem{FM} Fox, R.H and Milnor, J.W, singularities of 2-spheres in 4-space and cobordism of knots, Osaka J. Math, (3) (1966), 257-267.

\bibitem{F}  Fernandes, Alexandre Topological equivalence of complex curves and bi-Lipschitz homeomorphisms. Michigan Math. J. 51 (2003), no. 3, 593–606.

\bibitem{LD} O'Shea, Donal B.; Wilson, Leslie C, exceptional rays and bilipschitz geometry of real surface singularities https://arxiv.org/pdf/1705.05069.pdf. (2017)
\end{thebibliography}
\end{document}